\documentclass[12pt]{article}
\usepackage{paralist,amsmath,amsthm,amssymb,mathtools,url,graphicx,pdfpages,tikz,pdfpages,rotating}
\usepackage{cancel}
\usepackage[affil-it]{authblk}
\usepackage[left=1in,right=1in,top=1.2in, bottom=1in]{geometry}
\usepackage{mathdots}
\usepackage{bm} 

\newtheorem{theorem}{Theorem}[section]
\newtheorem{lemma}[theorem]{Lemma}
\newtheorem{corollary}[theorem]{Corollary}

\theoremstyle{definition}

\newtheorem{example}[theorem]{Example}

\numberwithin{equation}{section}

\DeclareMathOperator{\rank}{rank}

\newcommand{\disp}{\displaystyle}
 
\begin{document}
\title{Moore-Penrose Inverses of the Signless Laplacian and Edge-Laplacian of Graphs}

\author{Ryan Hessert} 
\author{Sudipta Mallik\thanks{Corresponding author}} 
\affil{\small Department of Mathematics and Statistics, Northern Arizona University, 801 S. Osborne Dr.\\ PO Box: 5717, Flagstaff, AZ 86011, USA 

rph53@nau.edu, sudipta.mallik@nau.edu}

\maketitle
\begin{abstract}
The signless Laplacian $Q$ and signless edge-Laplacian $S$ of a given graph may or may not be invertible. The Moore-Penrose inverses of $Q$ and $S$ are studied. In particular, using the incidence matrix, we find combinatorial formulas of the Moore-Penrose inverses of $Q$ and $S$ for trees. Also we present combinatorial formulas of the inverses of $Q$ and $S$ for odd unicyclic graphs.
\end{abstract}

\renewcommand{\thefootnote}{\fnsymbol{footnote}} 
\footnotetext{\emph{2010 Mathematics Subject Classification: 05C50, 15A09\\  
Keywords: Moore-Penrose inverse, Signless Laplacian, Incidence matrix,  Tree, Unicyclic graph}}
\renewcommand{\thefootnote}{\arabic{footnote}} 

\section{Introduction}
Let $G$ be a simple graph on $n$ vertices $1,2,\ldots,n$ and $m$ edges $e_1,e_2,\ldots,e_m$ with the adjacency matrix $A$ and the degree matrix $D$. The {\it signless Laplacian} $Q$ of $G$ is defined as $Q=D+A$. The vertex-edge {\it incidence matrix} $M$ of $G$ is the $n\times m$ matrix whose $(i,j)$-entry is $1$ if vertex $i$ is incident with edge $e_j$ and $0$ otherwise. It is well-known that $Q=MM^T$. The {\it signless edge-Laplacian} $S$ of $G$ is an $m\times m$ matrix whose rows and columns are indexed by edges of $G$ and $S$ is defined as $S=M^TM$.

\begin{theorem}\cite[Prop $2.1$]{CRC}\label{eigenvalue 0}
The smallest eigenvalue of the signless Laplacian of a connected graph is equal to $0$ if and only if the graph is bipartite. In this case $0$ is a simple eigenvalue.
\end{theorem}

By Theorem \ref{eigenvalue 0}, the signless Laplacian $Q$ of a connected bipartite graph such as a tree is not invertible. We investigate a generalized inverse of $Q$ and its combinatorial interpretations.

The {\it Moore–Penrose inverse} of an $m\times n$ real matrix $A$, denoted by $A^+$, is the $n\times m$ real matrix that satisfies the following equations \cite{BG}:
$$AA^+A=A, A^+AA^+=A^+, (AA^+)^T=AA^+, (A^+A)^T=A^+A.$$
If a square matrix $A$ is invertible, then $A^+=A^{-1}$.\\

The Moore-Penrose inverse of the incidence matrix of a graph was first studied by Ijira  in 1965 \cite{I}. Bapat studied the same for the Laplacian and edge-Laplacian of trees \cite{B}. Then several articles were published on the same topic for different graphs such as distance regular graphs \cite{AB,ABE}. Meanwhile the signless Laplacian of graphs has emerged to be an active area of research \cite{CRC, HM}. Then the study of the Moore-Penrose inverses of the signless Laplacian of graphs appeared to be an interesting new topic to investigate which led the authors to write this article.  In section 2, we find combinatorial formulas of the Moore-Penrose inverses of the signless Laplacian and edge-Laplacian for trees. In section 3, we do the same for odd unicyclic graphs. Distances among vertices and edges of a graph play a crucial rule in this study.\\

Throughout this article we use the following terminology: The vertices of a graph $G$ on $n$ vertices are labelled as $1,2,\ldots,n$. We write $|G|$ to denote the number of vertices of $G$. An edge $e_i$ of $G$ is written as $e_i=\{l_i,m_i\}$ where $l_i< m_i$. We denote $l_i\in V(G)$ or simply $l_i\in G$  and $e_i=\{l_i,m_i\}\in E(G)$ or simply $e_i=\{l_i,m_i\}\in G$ to state that $l_i$ is a vertex of $G$ and $e_i=\{l_i,m_i\}$ is an edge of $G$ respectively. The set of all edges of $G$ incident at a  vertex $v$ is denoted by $E(v).$  The {\it distance} between two vertices $l$ and $m$ of a connected graph $G$ is the number of edges in a shortest path between $l$ and $m$. We define the distance between vertex $j$ and edge $e_i=\{l_i,m_i\}$, denoted by $d(j,e_i)$ or $d(e_i,j)$, as follows:
$$d(j,e_i):=\min\{ d(j,l_i), d(j,m_i)\}.$$
We also define the distance $d(e_i,e_k)$ between edge $e_i=\{l_i,m_i\}$ and edge $e_k=\{l_k,m_k\}$ as follows:
$$d(e_i,e_k):=\min\{ d(l_i,e_k), d(m_i,e_k)\}.$$

\section{Moore-Penrose Inverses of the Signless Laplacian and Edge-Laplacian of Trees}
Let $G$ be a tree on $n$ vertices with an edge $e_i=\{l_i,m_i\}$  such that $l_i < m_i$. Then $G\setminus e_i$ has two connected components. The \textit{head component} of $G\setminus e_i$, denoted by $G_H(e_i)$, is the connected component of $G\setminus e_i$ containing vertex $m_i$. The \textit{tail component} of $G\setminus e_i$, denoted by $G_T(e_i)$, is the connected component of $G\setminus e_i$ containing vertex $l_i$. Note that $|G_H(e_i)|+|G_T(e_i)|=n$.

Now we introduce an $(n-1)\times n$ matrix $H$ whose rows and columns are indexed by the edges and vertices of tree $G$ respectively and $H=[h_{i,j}]$ is defined as follows:
\begin{equation}\label{H}
h_{i,j}=
 \frac{(-1)^{d(e_i,j)}}{n}\begin{cases} |G_T(e_i)| & \text{ if } j \in G_H(e_i)\\
|G_H(e_i)| & \text{ if } j \in G_T(e_i). 
\end{cases} 
\end{equation}

\begin{example}
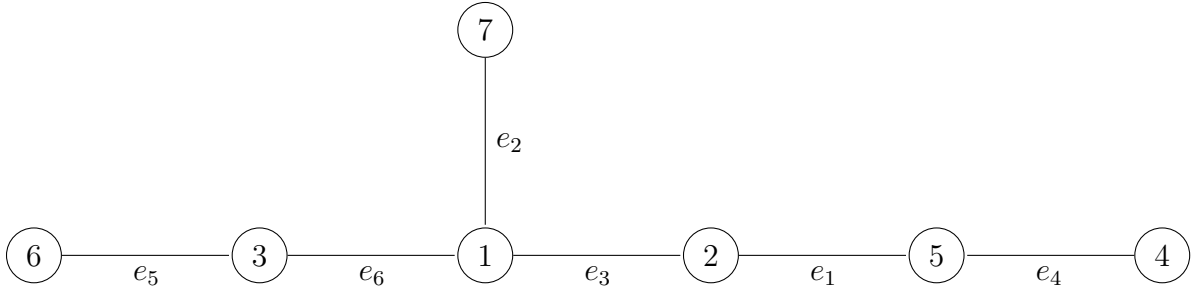
\begin{figure}
        \centering
        \begin{tikzpicture}[shorten > = 1pt, auto, node distance = .5cm]
\tikzset{vertex/.style = {shape = circle, draw, minimum size = 0.5em}}
\node[vertex] (6) at (-6,0) {6};
\node[vertex] (3) at (-3,0) {3};
\node[vertex] (1) at (0,0) {1};  
\node[vertex] (2) at (3,0) {2};
\node[vertex] (5) at (6,0) {5};
\node[vertex] (4) at (9,0) {4};
\node[vertex] (7) at (0,3) {7};
\draw (2) edge [below] node{$e_1$} (5);
\draw (7) edge [right] node{$e_2$} (1);
\draw (1) edge [below] node{$e_3$} (2);
\draw (4) edge [below] node{$e_4$} (5);
\draw (6) edge [below] node{$e_5$} (3);
\draw (3) edge [below] node{$e_6$} (1);
\end{tikzpicture}
        \caption{The smallest asymmetric tree}
        \label{fig:M tree}
    \end{figure}
    
For the tree given in Figure \ref{fig:M tree}, \\
$M= \left[ \begin{array}{rrrrrr}
    0&1&1&0&0&1\\
    1&0&1&0&0&0\\
    0&0&0&0&1&1\\
    0&0&0&1&0&0\\
    1&0&0&1&0&0\\
    0&0&0&0&1&0\\
    0&1&0&0&0&0\\
    \end{array} \right]$
    and $H= \disp\frac{1}{7} \left[ \begin{array}{rrrrrrr}
    -2&2&2&-5&5&-2&2\\
    1&-1&-1&-1&1&1&6\\
    3&4&-3&4&-4&3&-3\\
    1&-1&-1&6&1&1&-1\\
    -1&1&1&1&-1&6&1\\
    2&-2&5&-2&2&-5&-2
    \end{array} \right]$.
\end{example}

\begin{theorem}\label{HM=I}
Let $G$ be a tree on $n$ vertices with the incidence matrix $M$. For the matrix $H$ defined in (\ref{H}), we have $HM=I_{n-1}$.
\end{theorem}

\begin{proof}
For $i,j\in \{1,\ldots,n-1\}$, suppose $e_i=\{l,m\}$ where $l < m$ and $e_j=\{r,s\}$. Let $M=[m_{i,j}]$. Then the $(i,j)$-entry of $HM$ is given by
$$(HM)_{i,j}=\sum^{n}_{k=1} h_{i,k}m_{k,j} = h_{i,r} +h_{i,s}.$$

Case 1. $i=j$\\
In this case, $(HM)_{i,i}= h_{i,l} + h_{i,m}$. Note that $l\in G_T(e_i)$, $m \in G_H(e_i)$, and $d(e_i,l)=d(e_i,m)=0$. Then
$$(HM)_{i,i}=\frac{(-1)^{d(e_i,l)}}{n}|G_H(e_i)|+\frac{(-1)^{d(e_i,m)}}{n}|G_H(e_i)|=\frac{1}{n} |G_H(e_i)| +\frac{1}{n} |G_T(e_i)|=\frac{1}{n}n=1.$$ 

Case 2. $i\neq j$\\
Since $e_i \not= e_j$ and $G$ is a tree, either both $r,s\in G_H(e_i)$ or both $r,s\in G_T(e_i)$. Without loss of generality let $r,s\in G_H(e_i)$. Note that either $d(e_i,r)=d(m,r)=d(m,s)+1$ or $d(e_i,s)=d(m,s)=d(m,r)+1$. If $d(m,r)$ is even, then
$$(HM)_{i,j}= \frac{(-1)^{d(e_i,r)}}{n}|G_H(e_i)|+\frac{(-1)^{d(e_i,s)}}{n}|G_H(e_i)|=\frac{1}{n}|G_H(e_i)|-\frac{1}{n}|G_H(e_i)|=0.$$
If $d(m,r)$ is odd, then 
$$(HM)_{i,j}= \frac{(-1)^{d(e_i,r)}}{n}|G_H(e_i)|+\frac{(-1)^{d(e_i,s)}}{n}|G_H(e_i)|=-\frac{1}{n}|G_H(e_i)|+\frac{1}{n}|G_H(e_i)|=0.$$

Thus $HM=I_{n-1}$.
\end{proof}

For a vertex $i$ in a tree $G$, the set $E(i)$ of all edges of $G$ incident at $i$ are partitioned into two sets $E_H(i)$ and $E_T(i)$ defined as follows:
 $$E_H(i):=\{e_k\in E(G)\;|\; e_k \text{ is incident at } i \text{ and } i\in G_H(e_k)\},$$ 
 $$E_T(i):=\{e_k\in E(G)\;|\; e_k \text{ is incident at } i \text{ and }  i\in G_T(e_k)\}.$$
For an edge $e=\{l,m\}$ of a tree $G$, $G\setminus e[l]$ denotes the connected component of $G\setminus e$ that contains vertex $l$.
 
\begin{example}
For tree $G$ in Figure \ref{fig:M tree},
$$E_H(3)=\{e_6\},E_T(3)=\{e_5\},E_H(1)=\varnothing,E_T(1)=\{e_2,e_3,e_6\}.$$
Also $G\setminus e_6[3]$ and $G\setminus e_6[1]$ are shown in Figure \ref{subtree}.

\begin{figure}
        \centering
        \begin{tikzpicture}[shorten > = 1pt, auto, node distance = .5cm]
\tikzset{vertex/.style = {shape = circle, draw, minimum size = 0.5em}}
\node[vertex] (6) at (-6,0) {6};
\node[vertex] (3) at (-3,0) {3};
\node[vertex] (1) at (0,0) {1};  
\node[vertex] (2) at (3,0) {2};
\node[vertex] (5) at (6,0) {5};
\node[vertex] (4) at (9,0) {4};
\node[vertex] (7) at (0,3) {7};
\draw (2) edge [below] node{$e_1$} (5);
\draw (7) edge [right] node{$e_2$} (1);
\draw (1) edge [below] node{$e_3$} (2);
\draw (4) edge [below] node{$e_4$} (5);
\draw (6) edge [below] node{$e_5$} (3);
\node at (-4.5,-1){$G\setminus e_6[3]$};
\node at (4.5,-1){$G\setminus e_6[1]$};
\end{tikzpicture}
        \caption{$G\setminus e_6[3]$ and $G\setminus e_6[1]$ for tree $G$ in Figure \ref{fig:M tree}}
        \label{subtree}
    \end{figure}
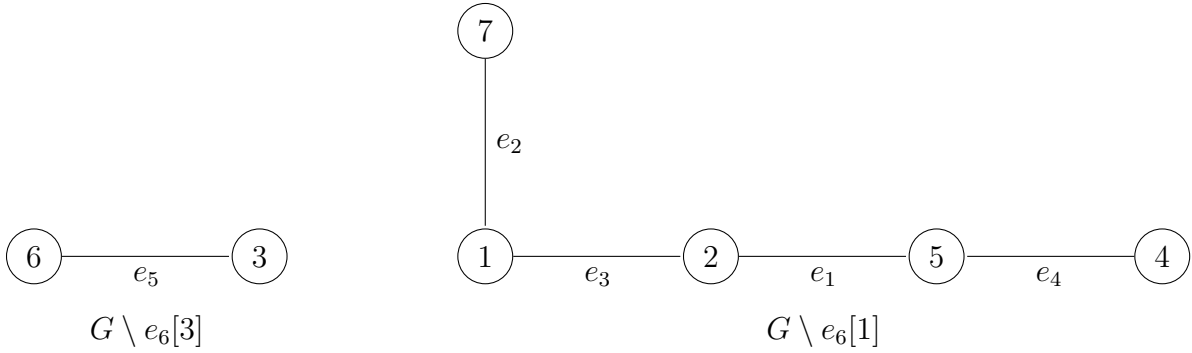
\end{example}

\begin{lemma}\label{tree sum = n-1}
Let $G$ be a tree on $n$ vertices and $i$ be a vertex $G$. Then 
$$\sum_{e_p \in E_H(i)} |G_T(e_{p})| + \sum_{e_q \in E_T(i)} |G_H(e_{q})| = n-1.$$
\end{lemma}
\begin{proof}
Suppose the vertices adjacent to $i$ are $l_1,l_2,\ldots,l_k$. For each $j=1,2,\ldots,k$, suppose $e_{p_j}$ is the edge between $i$ and $l_j$. Then vertex $i$  together with subtrees $G\setminus e_{p_j}[l_j]$, $j=1,2,\ldots,k$ form a spanning forest of $G$. Thus $$\sum_{e_{p_j} \in E(i)} |G\setminus e_{p_j}[l_j]|=n-1.$$
For each $j=1,2,\ldots,k$, $G\setminus e_{p_j}[l_j]$ is either $G_H(e_{p_j})$ or $G_T(e_{p_j})$. Thus 
$$\sum_{e_p \in E_H(i)} |G_T(e_{p})| + \sum_{e_q \in E_T(i)} |G_H(e_{q})|= \sum_{e_{p_j} \in E(i)} |G\setminus e_{p_j}[l_j]|=n-1.$$
\end{proof}

\begin{theorem}
Let $G$ be a tree on $n$ vertices with the incidence matrix $M$. The matrix $H$ defined in (\ref{H}) is the Moore-Penrose inverse of $M$. 
\end{theorem}
\begin{proof}
By Theorem \ref{HM=I}, $HM = I_{n-1}$. To prove $H=M^+$, it suffices to show that $MH$ is symmetric. Let $i, j\in \{1,\ldots,n\}$ and $i\neq j$. Then the $(i,j)$-entry of $MH$ is given by
$$(MH)_{i,j} 
=\sum_{p:e_p\in E(i)} h_{p,j}
=\sum_{t=1}^k h_{p_t,j} ,$$
where $e_{p_1},e_{p_2},\ldots,e_{p_k}$ are edges incident at $i$. Without loss of generality let $e_{p_1}$ be on the $i-j$ path. Note that $e_{p_1}\in E_H(i)$ if and only if $j\in G_T(e_{p_1})$. Also for $t=2,3,\ldots,k$, $e_{p_t}\in E_H(i)$ if and only if $j\in G_H(e_{p_t})$.

$$(MH)_{i,j} 
=\sum_{p:e_p\in E(i)} h_{p,j}
=\begin{cases}
h_{p_1,j}+\left(\disp\sum_{p:e_p\in E_H(i) \setminus e_{p_1}} h_{p,j}\right) + \left(\disp\sum_{q:e_q\in E_T(i)} h_{q,j}\right) & \text{if } e_{p_1}\in E_H(i)\\
h_{p_1,j}+\left(\disp\sum_{p:e_p\in E_H(i) } h_{p,j}\right) + \left(\disp\sum_{q:e_q\in E_T(i)\setminus e_{p_1}} h_{q,j}\right) & \text{if } e_{p_1}\in E_T(i).
\end{cases}$$

Case 1. $d(i,j)$ is even\\
Since $d(i,j)$ is even,  we have odd $d(e_{p_1},j)$ and even $d(e_{p_t},j)$ for $t=2,3,\ldots,k$.\\

Subcase (a). $e_{p_1}\in E_H(i)$
\begin{align*}
    &(MH)_{i,j}\\
=   &\frac{(-1)^{d(e_{p_1},j)}}{n}|G_H(e_{p_1})|+\left(\disp\sum_{p:e_p\in E_H(i) \setminus e_{p_1}} \frac{(-1)^{d(e_{p},j)}}{n}|G_T(e_{p})|\right) + \left(\disp\sum_{q:e_q\in E_T(i)} \frac{(-1)^{d(e_{q},j)}}{n}|G_H(e_{q})|\right) 
\end{align*}

Since $d(e_{p_1},j)$ is odd and $d(e_{p_t},j)$ is even for $t=2,3,\ldots,k$, we have
$$(MH)_{i,j}
=\frac{1}{n} \left( -|G_H(e_{p_1})|+\disp\sum_{p:e_p\in E_H(i) \setminus e_{p_1}} |G_T(e_{p})| + \disp\sum_{q:e_q\in E_T(i)} |G_H(e_{q})|\right).$$

By Lemma \ref{tree sum = n-1}, 
\begin{align*}
    (MH)_{i,j}&=\frac{1}{n} \left( \left( |G_T(e_{p_1})|-n\right) +\disp\sum_{p:e_p\in E_H(i) \setminus e_{p_1}} |G_T(e_{p})| + \left( n-1-\disp\sum_{p:e_p\in E_H(i)} |G_T(e_{p})|\right)\right)\\
    &=\frac{1}{n} \left( -1 +\disp\sum_{p:e_p\in E_H(i)} |G_T(e_{p})| -\disp\sum_{p:e_p\in E_H(i)} |G_T(e_{p})|\right)\\
    &=-\frac{1}{n}.
\end{align*}

Subcase (b). $e_{p_1}\in E_T(i)$
\begin{align*}
    &(MH)_{i,j}\\
=   &\frac{(-1)^{d(e_{p_1},j)}}{n}|G_T(e_{p_1})| + \left(\disp\sum_{p:e_p\in E_H(i)} \frac{(-1)^{d(e_{p},j)}}{n}|G_T(e_{p})| \right) 
+\left(\disp\sum_{q:e_q\in E_T(i) \setminus e_{p_1}} \frac{(-1)^{d(e_{q},j)}}{n}|G_H(e_{q})|\right)
\end{align*}

Since $d(e_{p_1},j)$ is odd and $d(e_{p_t},j)$ is even for $t=2,3,\ldots,k$, we have
$$(MH)_{i,j}
=\frac{1}{n} \left( -|G_T(e_{p_1})| + \disp\sum_{p:e_p\in E_H(i)} |G_T(e_{p})| +\disp\sum_{q:e_q\in E_T(i) \setminus e_{p_1}} |G_H(e_{q})|\right).$$

By Lemma \ref{tree sum = n-1}, 
\begin{align*}
    (MH)_{i,j}&=\frac{1}{n} \left( \left( |G_H(e_{p_1})|-n\right) 
    + \left( n-1-\disp\sum_{q:e_q\in E_T(i)} |G_H(e_{q})|\right)
    +\disp\sum_{p:e_p\in E_T(i) \setminus e_{p_1}} |G_H(e_{p})|\right)\\
    &=\frac{1}{n} \left( -1  -\disp\sum_{q:e_q\in E_T(i)} |G_H(e_{q})|
    +\disp\sum_{p:e_p\in E_T(i)} |G_H(e_{p})|\right)\\
    &=-\frac{1}{n}.
\end{align*}

By a proof similar to that in Subcases 1(a) and 1(b), we can show that $(MH)_{j,i}=-\frac{1}{n}$. Thus $$(MH)_{i,j}=(MH)_{j,i}=-\frac{1}{n}.$$ 

\vspace{12pt}
Case 2. $d(i,j)$ is odd\\
Since $d(i,j)$ is odd,  we have even $d(e_{p_1},j)$ and odd $d(e_{p_t},j)$ for $t=2,3,\ldots,k$.\\

Subcase (a). $e_{p_1}\in E_H(i)$\\
By a proof similar to that in  Subcase 1(a),
$$(MH)_{i,j}
=\frac{1}{n} \left( |G_H(e_{p_1})|-\disp\sum_{p:e_p\in E_H(i) \setminus e_{p_1}} |G_T(e_{p})| - \disp\sum_{q:e_q\in E_T(i)} |G_H(e_{q})|\right)=\frac{1}{n}.$$

Subcase (b). $e_{p_1}\in E_T(i)$\\
By a proof similar to that in  Subcase 1(b),
$$(MH)_{i,j}
=\frac{1}{n} \left( |G_T(e_{p_1})| - \disp\sum_{p:e_p\in E_H(i)} |G_T(e_{p})| -\disp\sum_{q:e_q\in E_T(i) \setminus e_{p_1}} |G_H(e_{q})|\right)=\frac{1}{n}.$$

By a proof similar to that in  Subcases 2(a) and 2(b), we can show that $(MH)_{j,i}=\frac{1}{n}$. Thus $$(MH)_{i,j}=(MH)_{j,i}=\frac{1}{n}.$$

\vspace{12pt}
From Cases 1 and 2, we conclude that $MH$ is a symmetric matrix.
\end{proof}

\begin{example}
For the tree given in Figure \ref{fig:M tree}, \\
$M= \left[ \begin{array}{rrrrrr}
    0&1&1&0&0&1\\
    1&0&1&0&0&0\\
    0&0&0&0&1&1\\
    0&0&0&1&0&0\\
    1&0&0&1&0&0\\
    0&0&0&0&1&0\\
    0&1&0&0&0&0\\
    \end{array} \right]$
    and $H=M^+= \disp\frac{1}{7} \left[ \begin{array}{rrrrrrr}
    -2&2&2&-5&5&-2&2\\
    1&-1&-1&-1&1&1&6\\
    3&4&-3&4&-4&3&-3\\
    1&-1&-1&6&1&1&-1\\
    -1&1&1&1&-1&6&1\\
    2&-2&5&-2&2&-5&-2
    \end{array} \right]$.\\
We have $M^+M=I_6$ 
and $MM^+= \disp\frac{1}{7} \left[ \begin{array}{rrrrrrr}
    6&1&1&1&-1&-1&1\\
    1&6&-1&-1&1&1&-1\\
    1&-1&6&-1&1&1&-1\\
    1&-1&-1&6&1&1&-1\\
    -1&1&1&1&6&-1&1\\
    -1&1&1&1&-1&6&1\\
    1&-1&-1&-1&1&1&6
    \end{array} \right].$
\end{example}  

Now we find combinatorial formulas for the Moore-Penrose inverses of the signless Laplacian $Q$ and the signless edge-Laplacian $S$ of a tree $G$. First recall that if $M$ is the incidence matrix of $G$, then $Q=MM^T$ and $S=M^TM$. Using $M^+$, we have the Moore-Penrose inverses of $Q$ and $S$ as follows:
$$Q^+=(MM^T)^+=(M^T)^+M^+=(M^+)^TM^+,$$  $$S^+=(M^TM)^+=M^+(M^T)^+=M^+(M^+)^T.$$

In the following results, a shortest path between vertices $i$ and $j$ in a connected graph $G$ is denoted by $P_{i-j}(G)$ or simply $P_{i-j}$.

\begin{lemma}\label{ij-path}
Let $G$ be a tree. Let $e_k$ be an edge and $i$ and $j$ be vertices of $G$. Then the following hold:
\begin{enumerate}
\item[(a)] $e_k$ is on the $ij$-path $P_{i-j}$ if and only if $(i,j)\in (G_H(e_k)\times G_T(e_k)) \cup (G_T(e_k)\times G_H(e_k))$.

\item[(b)] If $e_k$ is not on the $ij$-path $P_{i-j}$, then $d(i,j)$ and $d(e_k,i)+d(e_k,j)$ have the same parity.

\item[(c)] If $e_k$ is on the $ij$-path $P_{i-j}$, then $d(i,j)$ and $d(e_k,i)+d(e_k,j)$ have different parity.
\end{enumerate}
\end{lemma}
\begin{proof}
\begin{enumerate}
    \item[(a)]  Let $e_k \in P_{i-j}$. Note that $i \in G_H(e_k)$ if and only if  $j \in G_T(e_k)$. Thus 
    $$(i,j) \in (G_H(e_k) \times G_T(e_k)) \cup (G_T(e_k) \times G_H(e_k)).$$
    
    Conversely suppose that $(i,j) \in (G_H(e_k) \times G_T(e_k)) \cup (G_T(e_k) \times G_H(e_k)).$ This implies that $i \in G_H(e_k)$ if and only if  $j \in G_T(e_k)$. Then $e_k$ is on the $ij$-path $P_{i-j}$. 
    
    \item[(b)]  Suppose $e_k$ is not on the $ij$-path $P_{i-j}$. Then $d(e_k,i)=d(e_k,j)+d(i,j)$ or $d(e_k,j)=d(e_k,i)+d(i,j)$. Then $d(e_k,i)+d(e_k,j)$ is $2d(e_k,j)+d(i,j)$ or $2d(e_k,i)+d(i,j)$. Thus $d(i,j)$ and $d(e_k,i)+d(e_k,j)$ have the same parity.
    
    \item[(c)]  Suppose $e_k$ is on the $ij$-path $P_{i-j}$. Then $d(e_k,i)+d(e_k,j)=d(i,j)-1$. Thus $d(i,j)$ and $d(e_k,i)+d(e_k,j)$ have different parity.
\end{enumerate}
\end{proof}

\begin{theorem}
Let $G$ be a tree on $n$ vertices $1,2,\ldots,n$ with edges $e_1,e_2,\ldots,e_{n-1}$ and the signless Laplacian matrix $Q$. Then the Moore-Penrose inverse $Q^+=[q_{i,j}^+]$ of $Q$ is given by
$$q^+_{i,j}=\frac{(-1)^{d(i,j)}}{n^2}\left( \sum_{k:  i,j\in G_H(e_k)} |G_T(e_k)|^2
+\sum_{k: i,j\in G_T(e_k)} |G_H(e_k)|^2
-\sum_{k: e_k\in P_{i-j}} |G_H(e_k)||G_T(e_k)|  \right). $$
\end{theorem}

\begin{proof}
Since $M^+=H=[h_{i,j}]$ and $Q^+=(M^+)^TM^+$,
\begin{align*}
q^+_{i,j}&=\sum_{k:e_k\in E(G)} h_{k,i}h_{k,j} \\ 
&=\sum_{k:  i,j\in G_H(e_k)} h_{k,i}h_{k,j}
+\sum_{k: i,j\in G_T(e_k)} h_{k,i}h_{k,j}
+\sum_{k: e_k\in P_{i-j}} h_{k,i}h_{k,j}\\
&=\sum_{k:  i,j\in G_H(e_k)} \frac{(-1)^{d(e_k,i)+d(e_k,j)}}{n^2} |G_T(e_k)|^2
+\sum_{k: i,j\in G_T(e_k)}  \frac{(-1)^{d(e_k,i)+d(e_k,j)}}{n^2} |G_H(e_k)|^2\\
&\hspace{96pt}
+\sum_{k: e_k\in P_{i-j}}  \frac{(-1)^{d(e_k,i)+d(e_k,j)}}{n^2}  |G_H(e_k)||G_T(e_k)| \hspace{12pt} (\text{by Lemma } \ref{ij-path}(a))\\
&=\sum_{k:  i,j\in G_H(e_k)} \frac{(-1)^{d(i,j)}}{n^2} |G_T(e_k)|^2
+\sum_{k: i,j\in G_T(e_k)}  \frac{(-1)^{d(i,j)}}{n^2} |G_H(e_k)|^2\\
&\hspace{96pt}
+\sum_{k: e_k\in P_{i-j}}  \frac{-(-1)^{d(i,j)}}{n^2}  |G_H(e_k)||G_T(e_k)| \hspace{12pt} (\text{by Lemma } \ref{ij-path}(b)(c))\\
&=\frac{(-1)^{d(i,j)}}{n^2}\left( \sum_{k:  i,j\in G_H(e_k)} |G_T(e_k)|^2
+\sum_{k: i,j\in G_T(e_k)} |G_H(e_k)|^2
-\sum_{k: e_k\in P_{i-j}} |G_H(e_k)||G_T(e_k)|  \right).
\end{align*}

Note that the preceding formula is still true for $i=j$ and it becomes 
$$q^+_{i,i}=\frac{1}{n^2}\left( \sum_{k:  i\in G_H(e_k)} |G_T(e_k)|^2
+\sum_{k: i\in G_T(e_k)} |G_H(e_k)|^2  \right).$$
\end{proof}

\begin{example}
For the tree given in Figure \ref{fig:M tree},
$$Q= \left[ \begin{array}{rrrrrrr}
    3&1&1&0&0&0&1\\
    1&2&0&0&1&0&0\\
    1&0&2&0&0&1&0\\
    0&0&0&1&1&0&0\\
    0&1&0&1&2&0&0\\
    0&0&1&0&0&1&0\\
    1&0&0&0&0&0&1\\
    \end{array} \right]
\text{ and }
Q^+=\frac{1}{49} \left[ \begin{array}{rrrrrrr}
    20&1&-6&22&-15&-1&-13\\
    1&27&-15&6&-13&22&-8\\
    -6&-15&41&-36&29&-34&-1\\
    22&6&-36&83&-41&43&-29\\
    -15&-13&29&-41&48&-36&22\\
    -1&22&-34&43&-36&76&8\\
    -13&-8&-1&-29&22&8&55\\
    \end{array} \right].$$
\end{example}

For a tree $G$ with two distinct edges $e_i$ and $e_j$, $G\setminus \{e_i,e_j\}$ have three connected components. The component of  $G\setminus \{e_i,e_j\}$ that contains vertices incident with $e_i$ and $e_j$ is denoted by $G[e_i,e_j]$. The component of  $G\setminus \{e_i,e_j\}$ that contains a vertex incident with $e_i$ and no vertex incident with $e_j$ is denoted by $G[e_i,e_j)$. Similarly the component of  $G\setminus \{e_i,e_j\}$ that contains a vertex incident with $e_i$ and no vertex incident with $e_j$ is denoted by $G(e_i,e_j]$.

\begin{example}
For tree $G$ in Figure \ref{fig:M tree}, $G[e_6,e_1)$, $G[e_6,e_1]$, and $G(e_6,e_1]$ are given in Figure \ref{fig:G[e_i,e_j]}.

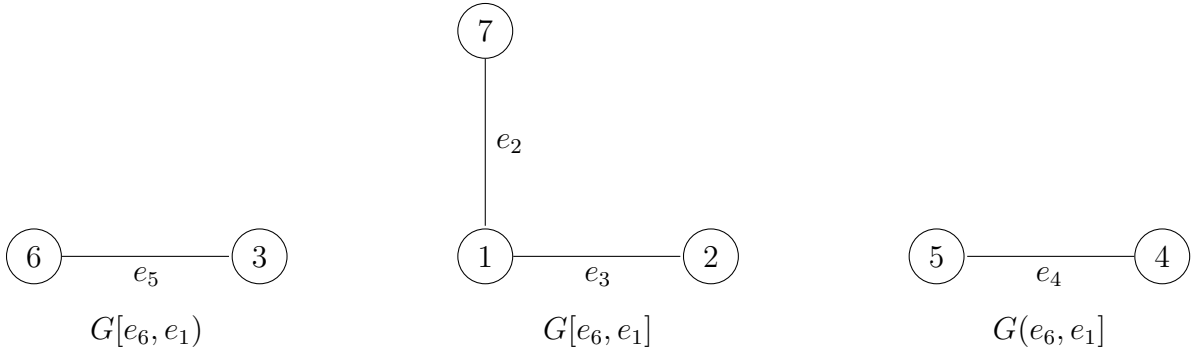
\begin{figure}
        \centering
        \begin{tikzpicture}[shorten > = 1pt, auto, node distance = .5cm]
\tikzset{vertex/.style = {shape = circle, draw, minimum size = 0.5em}}
\node[vertex] (6) at (-6,0) {6};
\node[vertex] (3) at (-3,0) {3};
\node[vertex] (1) at (0,0) {1};  
\node[vertex] (2) at (3,0) {2};
\node[vertex] (5) at (6,0) {5};
\node[vertex] (4) at (9,0) {4};
\node[vertex] (7) at (0,3) {7};
\draw (7) edge [right] node{$e_2$} (1);
\draw (1) edge [below] node{$e_3$} (2);
\draw (4) edge [below] node{$e_4$} (5);
\draw (6) edge [below] node{$e_5$} (3);
\node at (-4.5,-1){$G[e_6,e_1)$};
\node at (1.5,-1){$G[e_6,e_1]$};
\node at (7.5,-1){$G(e_6,e_1]$};
\end{tikzpicture}
        \caption{$G[e_6,e_1)$, $G[e_6,e_1]$, and $G(e_6,e_1]$ for tree $G$ in Figure \ref{fig:M tree}}
        \label{fig:G[e_i,e_j]}
    \end{figure}
\end{example}

\begin{lemma}\label{parity G[e_i,e_j]}
Let $G$ be a tree. Let $k$ be a vertex and $e_i$ and $e_j$ be distinct edges of $G$. Then the following hold:
\begin{enumerate}
\item[(a)] If $k$ is in $G[e_i,e_j]$, then $d(e_i,e_j)$ and $d(e_i,k)+d(e_j,k)$ have the same parity.

\item[(b)]  If $k$ is not in $G[e_i,e_j]$, then $d(e_i,e_j)$ and $d(e_i,k)+d(e_j,k)$ have different parity.
\end{enumerate}
\end{lemma}
\begin{proof}
\begin{enumerate}
    \item[(a)] Suppose $k$ is in $G[e_i,e_j]$. Suppose $d$ is the shortest distance between $k$ and a vertex on the path between $e_i$ and $e_j$. Then $d(e_i,k)+d(e_j,k)=2d+d(e_i,e_j)$. Thus $d(e_i,e_j)$ and $d(e_i,k)+d(e_j,k)$ have the same parity.

    \item[(b)] Suppose $k$ is not in $G[e_i,e_j]$. Then $d(e_i,e_j)+1$ is $d(e_i,k)-d(e_j,k)$ or $d(e_j,k)-d(e_i,k)$. So $d(e_i,k)+d(e_j,k)$ is $2d(e_i,k)+d(e_i,e_j)+1$ or $2d(e_j,k)+d(e_i,e_j)+1$. Thus $d(e_i,e_j)$ and $d(e_i,k)+d(e_j,k)$ have different parity.
\end{enumerate}
\end{proof}

\begin{theorem}
Let $G$ be a tree on $n$ vertices $1,2,\ldots,n$ with edges $e_1,e_2,\ldots,e_{n-1}$ and the incidence matrix $M$. Let $S=M^TM$ be the signless edge-Laplacian of $G$. Then the Moore-Penrose inverse $S^+=[s_{i,j}^+]$ of $S$ is given by
$$ s^+_{i,j}=\frac{(-1)^{d(e_i,e_j)}}{n}\begin{cases} 
|G_H(e_i)||G_T(e_i)| & \text{ if } e_i=e_j\\
-|G[e_i,e_j)| |G(e_i,e_j]| & \text{ if } e_i\neq e_j. \end{cases}$$
\end{theorem}
\begin{proof}
Since $M^+=H=[h_{i,j}]$ and $S^+=M^+(M^+)^T$, $s^+_{i,j}=\disp\sum_{k\in V(G)} h_{i,k}h_{j,k}$.  For $i=j$,
\begin{align*}
s^+_{i,i}&=\sum_{k\in V(G)} h_{i,k}^2 \\ 
&=\sum_{k\in G_H(e_i)} h_{i,k}^2  +\sum_{k\in G_T(e_i)} h_{i,k}^2 \\ 
&=\sum_{k\in G_H(e_i)} \frac{|G_T(e_i)|^2}{n^2}  +\sum_{k\in G_T(e_i)} \frac{|G_H(e_i)|^2}{n^2} \\
&=|G_H(e_i)|\frac{|G_T(e_i)|^2}{n^2}+|G_T(e_i)| \frac{|G_H(e_i)|^2}{n^2}\\
&=|G_H(e_i)||G_T(e_i)| \frac{|G_T(e_i)|+|G_H(e_i)|}{n^2}\\
&=\frac{|G_H(e_i)||G_T(e_i)|}{n} \hspace{12pt} (\text{since } |G_T(e_i)|  +|G_H(e_i)|=n).
\end{align*}

Let $i\neq j$. Then
$$s^+_{i,j}=\sum_{k\in V(G)} h_{i,k}h_{j,k} =\sum_{k\in G[e_i,e_j]} h_{i,k}h_{j,k}+\sum_{k\in G[e_i,e_j)} h_{i,k}h_{j,k}+\sum_{k\in G(e_i,e_j]} h_{i,k}h_{j,k}.$$

Case 1. $(e_i,e_j)\in G_H(e_j)\times G_H(e_i)$\\
This case is equivalent to having $G[e_i,e_j)=G_T(e_i)$ and $G(e_i,e_j]=G_T(e_j)$. Note that $|G[e_i,e_j]|=n-|G_T(e_i)|-|G_T(e_j)|=|G_H(e_i)|-|G_T(e_j)|$.

\begin{align*}
s^+_{i,j}&=\sum_{k\in G[e_i,e_j]} h_{i,k}h_{j,k} 
+\sum_{k\in G[e_i,e_j)} h_{i,k}h_{j,k}
+\sum_{k\in G(e_i,e_j]} h_{i,k}h_{j,k}\\
&=\sum_{k\in G[e_i,e_j]} \frac{(-1)^{d(e_i,k)+d(e_j,k)}}{n^2} |G_T(e_i)| |G_T(e_j)|
+\sum_{k\in G[e_i,e_j)} \frac{(-1)^{d(e_i,k)+d(e_j,k)}}{n^2} |G_H(e_i)| |G_T(e_j)|\\
& \hspace{155pt} +\sum_{k\in G(e_i,e_j]} \frac{(-1)^{d(e_i,k)+d(e_j,k)}}{n^2} |G_T(e_i)| |G_H(e_j)|\\
&=\sum_{k\in G[e_i,e_j]} \frac{(-1)^{d(e_i,e_j)}}{n^2} |G_T(e_i)| |G_T(e_j)|
+\sum_{k\in G[e_i,e_j)} \frac{-(-1)^{d(e_i,e_j)}}{n^2} |G_H(e_i)| |G_T(e_j)|\\
& \hspace{155pt} +\sum_{k\in G(e_i,e_j]} \frac{-(-1)^{d(e_i,e_j)}}{n^2} |G_T(e_i)| |G_H(e_j)| \hspace{12pt} (\text{by Lemma } \ref{parity G[e_i,e_j]})\\
&= \frac{(-1)^{d(e_i,e_j)}}{n^2} \Big( \big(\cancel{|G_H(e_i)|}-|G_T(e_j)| \big) |G_T(e_i)| |G_T(e_j)|
-\cancel{|G_T(e_i)| |G_H(e_i)| |G_T(e_j)|} \\
&  \hspace{278pt} -|G_T(e_j)| |G_T(e_i)| |G_H(e_j)| \Big) \\
&= \frac{(-1)^{d(e_i,e_j)}}{n^2} \Big( -|G_T(e_j)| |G_T(e_i)| |G_T(e_j)| -|G_T(e_j)| |G_T(e_i)| |G_H(e_j)|  \Big)\\
&= -\frac{(-1)^{d(e_i,e_j)}}{n^2} |G_T(e_i)||G_T(e_j)| \Big(  |G_T(e_j)|  +|G_H(e_j)| \Big)\\
&= -\frac{(-1)^{d(e_i,e_j)}}{n} |G_T(e_i)||G_T(e_j)| \hspace{12pt} (\text{since } |G_T(e_j)|  +|G_H(e_j)|=n)\\
&= -\frac{(-1)^{d(e_i,e_j)}}{n} |G[e_i,e_j)| |G(e_i,e_j]|.\\
\end{align*}

Case 2. $(e_i,e_j)\in G_T(e_j)\times G_T(e_i)$\\
This case is equivalent to having $G[e_i,e_j)=G_H(e_i)$ and $G(e_i,e_j]=G_H(e_j)$. By a proof similar to that in Case 1, we get 
$$s^+_{i,j}=-\frac{(-1)^{d(e_i,e_j)}}{n} |G_H(e_i)||G_H(e_j)|= -\frac{(-1)^{d(e_i,e_j)}}{n} |G[e_i,e_j)| |G(e_i,e_j]|.$$

Case 3. $(e_i,e_j)\in G_H(e_j)\times G_T(e_i)$\\
This case is equivalent to having $G[e_i,e_j)=G_H(e_i)$ and $G(e_i,e_j]=G_T(e_j)$. Note that $|G[e_i,e_j]|=n-|G_H(e_i)|-|G_T(e_j)|=|G_T(e_i)|-|G_T(e_j)|$. 

\begin{align*}
s^+_{i,j}&=\sum_{k\in G[e_i,e_j]} h_{i,k}h_{j,k} 
+\sum_{k\in G[e_i,e_j)} h_{i,k}h_{j,k}
+\sum_{k\in G(e_i,e_j]} h_{i,k}h_{j,k}\\
&=\sum_{k\in G[e_i,e_j]} \frac{(-1)^{d(e_i,k)+d(e_j,k)}}{n^2} |G_H(e_i)| |G_T(e_j)|
+\sum_{k\in G[e_i,e_j)} \frac{(-1)^{d(e_i,k)+d(e_j,k)}}{n^2} |G_T(e_i)| |G_T(e_j)|\\
& \hspace{155pt} +\sum_{k\in G(e_i,e_j]} \frac{(-1)^{d(e_i,k)+d(e_j,k)}}{n^2} |G_H(e_i)| |G_H(e_j)|\\
&=\sum_{k\in G[e_i,e_j]} \frac{(-1)^{d(e_i,e_j)}}{n^2} |G_H(e_i)| |G_T(e_j)|
+\sum_{k\in G[e_i,e_j)} \frac{-(-1)^{d(e_i,e_j)}}{n^2} |G_T(e_i)| |G_T(e_j)|\\
& \hspace{155pt} +\sum_{k\in G(e_i,e_j]} \frac{-(-1)^{d(e_i,e_j)}}{n^2} |G_H(e_i)| |G_H(e_j)| \hspace{12pt} (\text{by Lemma } \ref{parity G[e_i,e_j]})\\
&=\frac{(-1)^{d(e_i,e_j)}}{n^2} \Big( \big(\cancel{|G_T(e_i)|}-|G_T(e_j)| \big) |G_H(e_i)| |G_T(e_j)| - \cancel{|G_H(e_i)| |G_T(e_i)| |G_T(e_j)|} \\
&\hspace{278pt}  - |G_T(e_j)| |G_H(e_i)| |G_H(e_j)| \Big)\\
&=\frac{(-1)^{d(e_i,e_j)}}{n^2} \Big( -|G_T(e_j)| |G_H(e_i)| |G_T(e_j)| - |G_T(e_j)| |G_H(e_i)| |G_H(e_j)| \Big) \\
&=-\frac{(-1)^{d(e_i,e_j)}}{n^2}|G_H(e_i)| |G_T(e_j)|  \Big( |G_T(e_j)|  + |G_H(e_j)| \Big) \\
&=-\frac{(-1)^{d(e_i,e_j)}}{n}|G_H(e_i)| |G_T(e_j)| \hspace{12pt} (\text{since } |G_T(e_j)|  +|G_H(e_j)|=n)\\
&= -\frac{(-1)^{d(e_i,e_j)}}{n} |G[e_i,e_j)| |G(e_i,e_j]|.
\end{align*}

Case 4. $(e_i,e_j)\in G_T(e_j)\times G_H(e_i)$\\
This case is equivalent to having $G[e_i,e_j)=G_T(e_i)$ and $G(e_i,e_j]=G_H(e_j)$. By a proof similar to that in Case 3, we get 
$$s^+_{i,j}=-\frac{(-1)^{d(e_i,e_j)}}{n}|G_T(e_i)| |G_H(e_j)|
= -\frac{(-1)^{d(e_i,e_j)}}{n} |G[e_i,e_j)| |G(e_i,e_j]|.$$
\end{proof}

\begin{example}
For the tree given in Figure \ref{fig:M tree}, 
$$S= \left[ \begin{array}{rrrrrr}
    2&0&1&1&0&0\\
    0&2&1&0&0&1\\
    1&1&2&0&0&1\\
    1&0&0&2&0&0\\
    0&0&0&0&2&1\\
    0&1&1&0&1&2\\
    \end{array} \right]
\text{ and }
S^+=\frac{1}{7} \left[ \begin{array}{rrrrrrr}
    10&2&-8&-5&-2&4\\
    2&6&-3&-1&1&-2\\
    -8&-3&12&4&3&-6\\
    -5&-1&4&6&1&-2\\
    -2&1&3&1&6&-5\\
    4&-2&-6&-2&-5&10\\
    \end{array} \right].$$
\end{example}

We end this section by the following property of the Moore-Penrose inverse $M^+$ of the incidence matrix $M$ of a connected graph which is similar to \cite[Theorem 1]{I}. 

\begin{theorem}\label{M+}
Let $G$ be a connected graph on $n$ vertices $1,2,\ldots,n$ with the incidence matrix $M$. 
\begin{enumerate}
    \item[(a)] If $G$ has an odd cycle, then $MM^+=I_n$. 
    
    \item[(b)] If $G$ has no odd cycles (i.e., $G$ is bipartite), then $$MM^+=I_n-\frac{1}{n}[(-1)^{d(i,j)}].$$
\end{enumerate}
\end{theorem}

\begin{proof}
Since $MM^+M=M$, $(I_n-MM^+)M=O_{n,m}$. Then each row of $I_n-MM^+$ is orthogonal to each column of $M$. Suppose $x^T=[x_1,x_2,\ldots,x_n]$ is orthogonal to each column of $M$. Then $x_i+x_j=0$ for any edge $\{i,j\}$ in $G$. Therefore if there is a path between $i$ and $j$ of length $d$, then $x_i=(-1)^d x_j$.

(a) Suppose $G$ has an odd cycle $C$ and $v$ be a vertex on $C$. Let $k$ be another vertex on $C$. Then there is a path of odd length and a path of even length between $k$ and $v$. Then $x_v=-x_k$ and $x_v=x_k$ which implies $x_v=0$. Thus $x_v=0$ for any vertex $v$ on $C$. Since $G$ is connected, there is a path between a vertex $t$ not on $C$ and a vertex $v$ on $C$. Thus $x_t=\pm x_v=0$ for any vertex $t$ not on $C$. Thus  $x^T=0^T$ which implies $I_n-MM^+=O_n$, i.e., $MM^+=I_n$.

(b) Suppose $G$ has no odd cycles. Then $G$ is a connected bipartite graph and by Theorem \ref{eigenvalue 0}, $\rank(M)=\rank(MM^T)=\rank(Q)=n-1$. Also note that since $G$ has no odd cycles, lengths of the paths between two fixed vertices in $G$ are either all odd or all even. Then for a fixed vertex $i$ of $G$, $x_v=(-1)^{d(i,v)}x_i$ for any vertex $v$ of $G$. Thus $$[x_1,x_2,\ldots,x_n]=x_i[(-1)^{d(i,1)},(-1)^{d(i,2)},\ldots,(-1)^{d(i,n)}].$$
Consequently any row of $I_n-MM^+$ is a multiple of $[(-1)^{d(i,1)},(-1)^{d(i,2)},\ldots,(-1)^{d(i,n)}]$. 
 Suppose for each $i=1,2,\ldots,n$, $i$th row of  $I_n-MM^+$ is 
$$c_i[(-1)^{d(i,1)},(-1)^{d(i,2)},\ldots,(-1)^{d(1,n)}].$$
Note that if $c_1=c_2=\cdots=c_n=0$, then $MM^+=I_n$ and consequently $n=\rank(MM^+)\leq \rank(M)=n-1$, a contradiction. Thus $c_i\neq 0$ for some $i$ and by symmetry of $MM^+$, $c_i\neq 0$ for all $i=1,2,\ldots,n$. Now observe that  $(I_n-MM^+)^2=I_n-MM^+$ because $M^+MM^+=M^+$. Then the $(i,i)$-entry of $(I_n-MM^+)^2$, which is $c_i^2n$, is same as the $(i,i)$-entry of $I_n-MM^+$, which is $c_i(-1)^{d(i,i)}=c_i$. Thus $c_i=\frac{1}{n}$ and the $(i,j)$-entry of $I_n-MM^+$ is $\frac{(-1)^{d(i,j)}}{n}$.
\end{proof}

\begin{corollary}
Let $G$ be a connected bipartite graph on $n$ vertices and $m$ edges with the incidence matrix $M$.  Then 
$$M^+[(-1)^{d(i,j)}]=O_{m,n} \text{ and } [(-1)^{d(i,j)}]^2=n[(-1)^{d(i,j)}].$$
\end{corollary}
\begin{proof}
By Theorem \ref{M+}, $MM^+=I_n-\frac{1}{n}[(-1)^{d(i,j)}].$ Since $M^+=M^+MM^+$, we have
$$M^+=M^+MM^+=M^+\left(I_n-\frac{1}{n}[(-1)^{d(i,j)}]\right)=M^+-\frac{1}{n} M^+[(-1)^{d(i,j)}]$$
which implies 
$$M^+[(-1)^{d(i,j)}]=O_{m,n}.$$

Since $M^+[(-1)^{d(i,j)}]=O_{m,n}$, we have
$$MM^+[(-1)^{d(i,j)}]=O_n \implies \left(I_n-\frac{1}{n}[(-1)^{d(i,j)}]\right)[(-1)^{d(i,j)}]=O_n.$$
Thus we get $[(-1)^{d(i,j)}]^2=n[(-1)^{d(i,j)}]$.
\end{proof}

\section{Inverse of the Signless Laplacian and Edge-Laplacian of Odd Unicyclic Graphs}
A {\it unicyclic graph} on $n$ vertices is a simple connected graph that has $n$ edges. For a unicyclic graph $G$ with the cycle $C$ and an edge $e$ not in $C$, $G\setminus e$ has two connected components. The connected component of $G\setminus e$ that contains $C$ is denoted by $G\setminus e[C]$. Similarly the connected component of $G\setminus e$ that does not contain $C$ is denoted by $G\setminus e(C)$. When $e$ is on $C$, $G\setminus e[C]$ and $G\setminus e(C)$ are defined to be $G\setminus e$ and the empty graph respectively.  The unique shortest path between a vertex $i$ and $C$ is denoted by $P_{i-C}$.

\begin{theorem}
Let $G$ be an odd unicyclic graph on $n$ vertices $1,2,\ldots,n$ and edges $e_1,e_2,\ldots,e_n$ with the cycle $C$ and the incidence matrix $M$. Then $M$ is invertible and its inverse $M^{-1}=[a_{i,j}]$ is given by
\begin{equation}\label{unicyclic M inverse}
a_{i,j}= \begin{cases} 
\frac{(-1)^{d(e_i,j)}}{2}  & \text{ if } e_i \in C\\
0 & \text{ if } e_i \not \in C \text{ and } j \in G\setminus e_i [C]\\
(-1)^{d(e_i,j)}  & \text{ if } e_i \not \in C \text{ and } j \not \in G\setminus e_i [C].
\end{cases} 
\end{equation}
\end{theorem}

\begin{proof}
By Theorem \ref{eigenvalue 0}, $Q=MM^T$ is invertible and $Q$ has full rank. Since $\rank(M)=\rank(MM^T)$, $M$ has full rank and therefore $M$ is invertible. Consider $A=[a_{i,j}]$ as defined in (\ref{unicyclic M inverse}). Since $M=[m_{i,j}]$ is a square matrix, it suffices to show that $MA=I_{n}$. The $(i,j)$-entry of $MA$ is
$$(MA)_{i,j}=\sum_{k:e_k\in E(G)} m_{i,k} a_{k,j}=\sum_{k:e_k\in E(i)} a_{k,j}.$$

For $i=j \notin C$, 
$$(MA)_{i,j}=\sum_{k:e_k\in E(i)} a_{k,j}
=\sum_{k:e_k\in E(i),\; i\notin G\setminus e_i [C]} 1+ \sum_{k:e_k\in E(i),\; i\in G\setminus e_i [C]} 0=1.$$

For $i=j \in C$,
$$(MA)_{i,j}=\sum_{k:e_k\in E(i)} a_{k,j}
=\sum_{k:e_k\in E(i)\cap E(C)} \frac{1}{2}+ \sum_{k:e_k\in E(i)\setminus E(C)} 0=2\cdot \frac{1}{2}=1.$$

\noindent Suppose $i\neq j$.\\

Case 1. $i,j\notin C$
$$(MA)_{i,j}=\sum_{k:e_k\in E(i)} a_{k,j}
=\sum_{\substack{k:e_k \in E(i)\setminus E(C) \\ j \in G\setminus{e_k}[C]}} 0 + \sum_{\substack{k:e_k \in E(i)\setminus E(C) \\ j \not \in G\setminus{e_k}[C]}} (-1)^{d(e_k,j)}.$$
There are two edges, say $e_{k_1}$ and $e_{k_2}$, for which $e_k \in E(i)\setminus E(C) $ and $j \not \in G\setminus{e_k}[C]$. Note that $i\in P_{j-C}$ and $e_{k_1}$ and $e_{k_2}$ are on $P_{j-C}$. Thus $d(e_{k_1},j)$ and $d(e_{k_2},j)$ have different parity and consequently $(MA)_{i,j}=(-1)^{d(e_{k_1},j)} +(-1)^{d(e_{k_2},j)} =0$.\\

Case 2. $i,j\in C$
$$(MA)_{i,j}=\sum_{k:e_k\in E(i)} a_{k,j}
=\sum_{\substack{k:e_k \in E(i)\setminus E(C) \\ j \in G\setminus{e_k}[C]}} 0 + \sum_{k:e_k \in E(i)\cap E(C)} \frac{(-1)^{d(e_k,j)}}{2}.$$
There are two edges, say $e_{k_1}$ and $e_{k_2}$, that are in $E(i)\cap E(C)$. Note that $d(e_{k_1},j)$ and $d(e_{k_2},j)$ have different parity and consequently $(MA)_{i,j}=\frac{(-1)^{d(e_{k_1},j)}}{2} +\frac{(-1)^{d(e_{k_2},j)}}{2} =0$.\\

Case 3. $i\in C$ and $j\notin C$.
\begin{align*}
(MA)_{i,j}&=\sum_{\substack{k:e_k \in E(i)\setminus E(C) \\ j \in G\setminus{e_k}[C]}} 0 
+ \sum_{\substack{k:e_k \in E(i)\setminus E(C) \\ j \notin G\setminus{e_k}[C]}} (-1)^{d(e_k,j)}
+\sum_{k:e_k\in E(i)\cap E(C)} \frac{(-1)^{d(e_k,j)}}{2}\\
&=\sum_{\substack{k:e_k \in E(i)\setminus E(C) \\ j \notin G\setminus{e_k}[C]}} (-1)^{d(e_k,j)}
+\sum_{k:e_k\in E(i)\cap E(C)} \frac{(-1)^{d(e_k,j)}}{2}.
\end{align*}
There are two edges, say $e_{k_1}$ and $e_{k_2}$, that are in $E(i)\cap E(C)$. Now either $i\in P_{j-C}$ or $i\notin P_{j-C}$.
First suppose $i\in P_{j-C}$. Then there is a unique edge, say $e_t$, in $E(i)\setminus E(C)$ such that $j \notin G\setminus{e_k}[C]$. Note that $d(e_t,j)=d(i,j)-1$ and $d(e_{k_1},j)=d(e_{k_2},j)=d(i,j)$. Then
$$(MA)_{i,j}=(-1)^{d(e_t,j)}+\frac{(-1)^{d(e_{k_1},j)}}{2}+\frac{(-1)^{d(e_{k_2},j)}}{2}
=(-1)^{d(i,j)-1}+\frac{(-1)^{d(i,j)}}{2}+\frac{(-1)^{d(i,j)}}{2}=0.$$

Now suppose $i\notin P_{j-C}$. First observe that there is no edge in $e_k \in E(i)\setminus E(C)$ for which $j \notin G\setminus{e_k}[C]$. We have either $d(e_{k_1},j)=1+d(e_{k_2},j)$ or $d(e_{k_2},j)=1+d(e_{k_1},j)$. Thus
$$(MA)_{i,j}=\frac{(-1)^{d(e_{k_1},j)}}{2}+\frac{(-1)^{d(e_{k_2},j)}}{2}=0.$$

Case 4. $i\notin C$ and $j\in C$\\
This case has a proof similar to that of Case 3.
\end{proof}

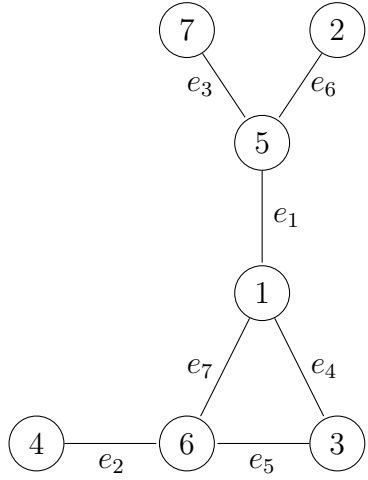
\begin{figure}
        \centering
        \begin{tikzpicture}[shorten > = 1pt, auto, node distance = .5cm ]
\tikzset{vertex/.style = {shape = circle, draw, minimum size = 1em}}
\tikzset{edge/.style = {-}}
\node[vertex] (1) at (-1,4.5){$7$};
\node[vertex] (2) at (1,4.5){$2$};
\node[vertex] (3) at (0,3){$5$};
\node[vertex] (4) at (0,1){$1$};
\node[vertex] (5) at (1,-1){$3$};
\node[vertex] (6) at (-1,-1){$6$};
\node[vertex] (7) at (-3,-1){$4$};
\draw[edge] (1) edge node[left]{$e_3$} (3);
\draw[edge] (2) edge node[right]{$e_6$} (3);
\draw[edge] (3) edge node[right]{$e_1$} (4);
\draw[edge] (4) edge node[right]{$e_4$} (5);
\draw[edge] (5) edge node[below]{$e_5$} (6);
\draw[edge] (7) edge node[below]{$e_2$} (6);
\draw[edge] (4) edge node[left]{$e_7$} (6);
\end{tikzpicture}
        \caption{An odd unicyclic graph}
        \label{fig:odd unicyclic}
    \end{figure}

\begin{example}
For the odd unicyclic graph given in Figure \ref{fig:odd unicyclic}, 
$$M= \left[\begin{array}{rrrrrrr}
1 & 0 & 0 & 1 & 0 & 0 & 1 \\
0 & 0 & 0 & 0 & 0 & 1 & 0 \\
0 & 0 & 0 & 1 & 1 & 0 & 0 \\
0 & 1 & 0 & 0 & 0 & 0 & 0 \\
1 & 0 & 1 & 0 & 0 & 1 & 0 \\
0 & 1 & 0 & 0 & 1 & 0 & 1 \\
0 & 0 & 1 & 0 & 0 & 0 & 0
\end{array}\right] 
\text{ and }
M^+= \left[\begin{array}{rrrrrrr}
0 & -1 & 0 & 0 & 1 & 0 & -1 \\
0 & 0 & 0 & 1 & 0 & 0 & 0 \\
0 & 0 & 0 & 0 & 0 & 0 & 1 \\
\frac{1}{2} & \frac{1}{2} & \frac{1}{2} & \frac{1}{2} &
-\frac{1}{2} & -\frac{1}{2} & \frac{1}{2} \\
-\frac{1}{2} & -\frac{1}{2} & \frac{1}{2} & -\frac{1}{2}
& \frac{1}{2} & \frac{1}{2} & -\frac{1}{2} \\
0 & 1 & 0 & 0 & 0 & 0 & 0 \\
\frac{1}{2} & \frac{1}{2} & -\frac{1}{2} & -\frac{1}{2}
& -\frac{1}{2} & \frac{1}{2} & \frac{1}{2}
\end{array}\right].$$
\end{example}

For a unicyclic graph with the cycle $C$,  the distance $d(i,C)$ between $i$ and $C$ is defined by $d(i,C):=|P_{i-C}|-1$. We write $i^*$ for the vertex on $C$ that is closest to $i$. Then $i^*=i$ when $i$ is on $C$. Also $P_{i-C}=P_{i-i^*}$. 

\begin{example}
Consider the odd unicyclic graph given in Figure \ref{fig:odd unicyclic}. Note that  $7^*=1$, $P_{7-C}=P_{7-7^*}=P_{7-1}$, and $d(7,C)=|P_{7-C}|-1=2$. 
\end{example}

\begin{lemma}\label{dist in cycle}
Let $G$ be an odd unicyclic graph on $n$ vertices $1,2,\ldots,n$ and edges $e_1,e_2,\ldots,e_n$ with the cycle $C$. Then $e_k\notin C$ and distinct $i,j\notin G\setminus e_k [C]$ if and only if $e_k\in P_{i-i^*}\cap P_{j-j^*}$. Moreover,
\begin{enumerate}
    \item[(a)] if $e_k\in C$ and $e_k\in P_{i-j}$, then $d(e_k,i)+d(e_k,j)$ and $d(i,j)$ have different parity,
    \item[(b)] if $e_k\in C$ and $e_k\notin P_{i-j}$, then $d(e_k,i)+d(e_k,j)$ and $d(i,j)$ have the same parity,
    \item[(c)] if $e_k\in P_{i-i^*}\cap P_{j-j^*}$, then $d(e_k,i)+d(e_k,j)$ and $d(i,j)$ have the same parity.
\end{enumerate}
\end{lemma}

\begin{proof}
The first part being clear, we proceed to the rest.

\begin{enumerate}
    \item[(a)] Suppose $e_k\in C$ and $e_k\in P_{i-j}$. Then
$$d(e_k,i)+d(e_k,j)=d(i,j)-1.$$
Thus $d(e_k,i)+d(e_k,j)$ and $d(i,j)$ have different parity.

\item[(b)] Suppose $e_k\in C$ and $e_k\notin P_{i-j}$. First suppose $P_{i-e_k}\cap P_{j-e_k}\neq \varnothing$. Then $d(e_k,i)+d(e_k,j)$ is $2d(e_k,i^*)+d(i,j)$ or $2d(e_k,j^*)+d(i,j)$. Then $d(e_k,i)+d(e_k,j)$ and $d(i,j)$ have the same parity.

Now suppose $P_{i-e_k}\cap P_{j-e_k}= \varnothing$. Then $d(i,i^*)+d(j,j^*)=d(i,j)-d(i^*,j^*)$ and $d(e_k,i^*)+d(e_k,j^*)=|C|-1-d(i^*,j^*)$. Thus 
\begin{align*}d(e_k,i)+d(e_k,j)
&=d(i,i^*)+d(i^*,e_k)+d(j,j^*)+d(j^*,e_k)\\
&=\big(d(i,i^*)+d(j,j^*)\big) +\big( d(i^*,e_k)+d(j^*,e_k)\big) & \\
&=d(i,j)-d(i^*,j^*)+ |C|-1-d(i^*,j^*)\\
&=d(i,j)+|C|-1-2d(i^*,j^*).\end{align*}
Thus $d(e_k,i)+d(e_k,j)$ and $d(i,j)$ have the same parity.

\item[(c)] Suppose $e_k\in P_{i-i^*}\cap P_{j-j^*}$. Let $t$ be the vertex on the path $P_{i-j}$ that is closest to $e_k$. Then 
$$d(e_k,i)+d(e_k,j)=2d(e_k,t)+d(i,j).$$
Thus $d(e_k,i)+d(e_k,j)$ and $d(i,j)$ have the same parity.
\end{enumerate}
\end{proof}

Using $M^{-1}$, we have the inverses of $Q$ and $S$ as follows:
$Q^{-1}=(M^{-1})^TM^{-1}$ and  $S^{-1}=M^{-1}(M^{-1})^T$. 

\begin{theorem}
Let $G$ be an odd unicyclic graph on $n$ vertices $1,2,\ldots,n$ with the cycle $C$ and the signless Laplacian matrix $Q$. Then the inverse $Q^{-1}=[q_{i,j}^+]$ of $Q$ is given by
$$q^+_{i,j}= \frac{(-1)^{d(i,j)}}{4}  \Big( |C|-2d(i^*,j^*)+ 4|E(P_{i-i^*}\cap P_{j-j^*})| \Big),$$
where $i^*$ and $j^*$ are the vertices on $C$  closest to $i$ and $j$ respectively.
\end{theorem}

\begin{proof}
Since $M^{-1}=[a_{i,j}]$ and $Q^{-1}=(M^{-1})^TM^{-1}$, $q^+_{i,j}=\sum_{k:e_k\in E(G)} a_{k,i}a_{k,j}$. For $i=j$,
\begin{align*}
q^+_{i,i}&=\sum_{k:e_k\in E(G)} a_{k,i}^2 \\ 
&=\sum_{k:e_k\in C} a_{k,i}^2
+\sum_{\substack{k:e_k\notin C \\ i\notin G\setminus e_k[C]}} a_{k,i}^2
+\sum_{\substack{k:e_k\notin C \\ i\in G\setminus e_k[C]}} a_{k,i}^2\\
&=\sum_{k:e_k\in C} a_{k,i}^2
+\sum_{\substack{k:e_k\notin C \\ i\notin G\setminus e_k[C]}} a_{k,i}^2.
\end{align*}

When $i\in C$,
$$q^+_{i,i}=\sum_{k:e_k\in C} a_{k,i}^2
+\sum_{\substack{k:e_k\notin C \\ i\notin G\setminus e_k[C]}} a_{k,i}^2
=\sum_{k:e_k\in C} \frac{1}{4}
=\frac{|C|}{4}.$$

When $i\notin C$,
$$q^+_{i,i}=\sum_{k:e_k\in C} a_{k,i}^2
+\sum_{\substack{k:e_k\notin C \\ i\notin G\setminus e_k[C]}} a_{k,i}^2
=\sum_{k:e_k\in C} \frac{1}{4}
+\sum_{k:e_k\in P_{i-C}} 1
=\frac{|C|}{4} +|E(P_{i-i^*})|.
$$

Since $d(i,i)=d(i^*,i^*)=0$ and $E(P_{i-i^*}\cap P_{i-i^*})=E(P_{i-i^*})$, we have 
$$q^+_{i,i}= \frac{(-1)^{d(i,i)}}{4}  \Big( |C|-2d(i^*,i^*)+ 4|E(P_{i-i^*}\cap P_{i-i^*})| \Big).$$

\noindent Suppose $i\neq j$.
\begin{align*}
q^+_{i,j}&=\sum_{k:e_k\in E(G)} a_{k,i}a_{k,j} \\ 
&=\sum_{\substack{k:e_k\in C \\ e_k\in P_{i-j}}} a_{k,i}a_{k,j}
+\sum_{\substack{k:e_k\in C \\ e_k\notin P_{i-j}}} a_{k,i}a_{k,j}
+\sum_{\substack{k:e_k\notin C\\ i,j\notin G\setminus e_i [C]}} a_{k,i}a_{k,j}\\
&=\sum_{\substack{k:e_k\in C \\ e_k\in P_{i-j}}} \frac{(-1)^{d(e_k,i)+d(e_k,j)}}{4}
+\sum_{\substack{k:e_k\in C \\ e_k\notin P_{i-j}}} \frac{(-1)^{d(e_k,i)+d(e_k,j)}}{4}
+\sum_{\substack{k:e_k\notin C\\ i,j\notin G\setminus e_i [C]}} (-1)^{d(e_k,i)+d(e_k,j)}\\
&=-\sum_{\substack{k:e_k\in C \\ e_k\in P_{i-j}}} \frac{(-1)^{d(i,j)}}{4}
+\sum_{\substack{k:e_k\in C \\ e_k\notin P_{i-j}}} \frac{(-1)^{d(i,j)}}{4}
+\sum_{k:e_k\in E(P_{i-i^*}\cap P_{j-j^*})} (-1)^{d(i,j)} \hspace{34pt}(\text{by Lemma }\ref{dist in cycle})\\
&=-|E(P_{i^*-j^*})|\frac{(-1)^{d(i,j)}}{4}
+|E(C\setminus P_{i^*-j^*})|\frac{(-1)^{d(i,j)}}{4}
+|E(P_{i-i^*}\cap P_{j-j^*}))| (-1)^{d(i,j)} \\
&=(-1)^{d(i,j)}  
\left( -\frac{d(i^*,j^*)}{4}+\frac{|C|-d(i^*,j^*)}{4}+|E(P_{i-i^*}\cap P_{j-j^*}))|   \right)\\
&= \frac{(-1)^{d(i,j)}}{4}  \Big( |C|-2d(i^*,j^*)+ 4|E(P_{i-i^*}\cap P_{j-j^*})| \Big).
\end{align*}

\end{proof}

\begin{corollary}
Let $G$ be an odd unicyclic graph on $n$ vertices $1,2,\ldots,n$ with the cycle $C$. Suppose $Q$ is the signless Laplacian matrix of $G$ and its inverse is $Q^{-1}=[q_{i,j}^+]$. Then vertex $i$ is in $C$ if and only if $q_{i,i}^+=\frac{|C|}{4}$.
\end{corollary}

\begin{example}
For the odd unicyclic graph given in Figure \ref{fig:odd unicyclic}, 
$$Q= \left[\begin{array}{rrrrrrr}
3 & 0 & 1 & 0 & 1 & 1 & 0 \\
0 & 1 & 0 & 0 & 1 & 0 & 0 \\
1 & 0 & 2 & 0 & 0 & 1 & 0 \\
0 & 0 & 0 & 1 & 0 & 1 & 0 \\
1 & 1 & 0 & 0 & 3 & 0 & 1 \\
1 & 0 & 1 & 1 & 0 & 3 & 0 \\
0 & 0 & 0 & 0 & 1 & 0 & 1
\end{array}\right] 
\text{ and }
Q^{-1}= \frac{1}{4} \left[\begin{array}{rrrrrrr}
3 & 3 & -1 & 1 & -3 & -1 & 3 \\
3 & 11 & -1 & 1 & -7 & -1 & 7 \\
-1 & -1 & 3 & 1 & 1 & -1 & -1 \\
1 & 1 & 1 & 7 & -1 & -3 & 1 \\
-3 & -7 & 1 & -1 & 7 & 1 & -7 \\
-1 & -1 & -1 & -3 & 1 & 3 & -1 \\
3 & 7 & -1 & 1 & -7 & -1 & 11
\end{array}\right].$$
\end{example}

For an odd unicyclic graph $G$ with the cycle $C$ and for distinct edges $e_i$ and $e_j$ on $C$, $P_{e_i-e_j}$ denotes the shortest path between vertices of edges $e_i$ and $e_j$ and $G\setminus{\{e_i,e_j\}}[P_{e_i-e_j}]$ denotes the connected component of $G\setminus{\{e_i,e_j\}}$ that contains $P_{e_i-e_j}$.

\begin{lemma}\label{parity of e_i,e_j}
Let $G$ be an odd unicyclic graph on $n$ vertices $1,2,\ldots,n$ and edges $e_1,e_2,\ldots,e_n$ with the cycle $C$. For distinct $e_i$ and $e_j$, the following hold:
\begin{enumerate}
    \item[(a)] If $e_i,e_j \in C$ and $k \in G\setminus{\{e_i,e_j\}}[P_{e_i-e_j}]$, then $d(e_i,k)+d(e_j,k)$ and $d(e_i,e_j)$ have the same parity.
    
    \item[(b)] If $e_i,e_j \in C$ and $k \notin G\setminus{\{e_i,e_j\}}[P_{e_i-e_j}]$, then $d(e_i,k)+d(e_j,k)$ and $d(e_i,e_j)$ have different parity. 
    
    \item[(c)] If $e_i,e_j \notin C$ and $k\in \big(G\setminus e_i(C))\cap (G\setminus e_j(C)\big)$, then $d(e_i,k)+d(e_j,k)$ and $d(e_i,e_j)$ have different parity. 
    
    \item[(d)] If $e_i \in C$, $e_j \notin C$, and $k\notin G\setminus e_j[C]$, then $d(e_i,k)+d(e_j,k)$ and $d(e_i,e_j)$ have different parity. 
\end{enumerate}
\end{lemma}
\begin{proof}
\begin{enumerate}
\item[(a)] Suppose $e_i,e_j \in C$ and $k \in G\setminus{\{e_i,e_j\}}[P_{e_i-e_j}]$. Then $$d(e_i,k)+d(e_j,k)=2d(k,k^*)+d(e_i,e_j).$$
Thus $d(e_i,k)+d(e_j,k)$ and $d(e_i,e_j)$ have the same parity.

\item[(b)] Suppose $e_i,e_j \in C$ and $k \notin G\setminus{\{e_i,e_j\}}[P_{e_i-e_j}]$. When $E(P_{k^*-e_i})\cap E(P_{k^*-e_j}) = \varnothing$,
\begin{align*}d(e_i,k)+d(e_j,k)&=2d(k,k^*)+d(e_i,k^*)+d(e_j,k^*)\\
&=2d(k,k^*)+|C|-2-d(e_i,e_j).\end{align*}

When $E(P_{k^*-e_i})\cap E(P_{k^*-e_j}) \neq \varnothing$, 
$d(e_i,k)+d(e_j,k)=2d(k,k^*)+d(e_i,k^*)+d(e_j,k^*)$
which is $2d(k,k^*)+2d(e_i,k^*)+1+d(e_i,e_j)$ or $2d(k,k^*)+2d(e_j,k^*)+1+d(e_i,e_j)$.

Thus $d(e_i,k)+d(e_j,k)$ and $d(e_i,e_j)$ have different parity. 

\item[(c)] Suppose $e_i,e_j \notin C$ and $k\in \big(G\setminus e_i(C))\cap (G\setminus e_j(C)\big)$. Note that $P_{e_i-i^*} \subseteq P_{e_j-j^*}$ or $P_{e_j-j^*} \subseteq P_{e_i-i^*}$. Thus $d(e_i,k)+d(e_j,k)$ is $2d(e_j,k)+1+d(e_i,e_j)$ or $2d(e_i,k)+1+d(e_i,e_j)$. Thus $d(e_i,k)+d(e_j,k)$ and $d(e_i,e_j)$ have different parity.

\item[(d)] Suppose $e_i \in C$, $e_j \notin C$, and $k\notin G\setminus e_j[C]$. Then $$d(e_i,k)+d(e_j,k)=2d(e_j,k)+1+d(e_i,e_j).$$
Thus $d(e_i,k)+d(e_j,k)$ and $d(e_i,e_j)$ have different parity.
\end{enumerate}
\end{proof}

\begin{theorem}
Let $G$ be an odd unicyclic graph on $n$ vertices $1,2,\ldots,n$ with edges $e_1,e_2,\ldots,e_n$ and the incidence matrix $M$. Let $S=M^TM$ be the signless edge-Laplacian of $G$. Then the inverse $S^{-1}=[s_{i,j}^+]$ of $S$ is given by
$$ s^+_{i,j}=\frac{(-1)^{d(e_i,e_j)}}{4}\begin{cases} 
n                       & \text{ if }  i=j,\; e_i\in C\\
4|G\setminus e_i(C)|    & \text{ if }  i=j,\; e_i\notin C\\
-4|\big(G\setminus e_i(C))\cap (G\setminus e_j(C)\big)| & \text{ if } i\neq j,\; e_i,e_j\notin C\\
2|G\setminus \{e_i,e_j\}[P_{e_i-e_j}]|-n & \text{ if } i\neq j,\; e_i,e_j\in C\\
-2|G\setminus e_i(C)|-2|G\setminus e_j(C)| & \text{ otherwise.} 
\end{cases}$$
\end{theorem}

\begin{proof}
Since $M^{-1}=[a_{i,j}]$ and $S^{-1}=M^{-1}(M^{-1})^T$, $s^+_{i,j}=\sum_{k\in V(G)} a_{i,k}a_{j,k}$.  For $e_i=e_j \in C$, 
$$s^+_{i,i}=\sum_{k\in V(G)} a_{i,k}^2=\sum_{k\in V(G)} \frac{1}{4}=\frac{n}{4}.$$

For $e_i=e_j \notin C$, 
$$s^+_{i,i}=\sum_{k\in V(G)} a_{i,k}^2=\sum_{k\in G\setminus e_i[C]} a_{i,k}^2 +\sum_{k\notin G\setminus e_i[C]}a_{i,k}^2 =0+\sum_{k\notin G\setminus e_i[C]} 1= |G\setminus e_i(C)|.$$

\noindent Suppose $e_i\neq e_j$.\\

Case 1. $e_i,e_j \in C$
\begin{align*}
s^+_{i,j}&=\sum_{k\in V(G)} a_{i,k}a_{j,k} \\ 
&=\sum_{k \in G\setminus{\{e_i,e_j\}}[P_{e_i-e_j}]} \frac{(-1)^{d(e_i,k)+d(e_j,k)}}{4}
+\sum_{k \notin G\setminus{\{e_i,e_j\}}[P_{e_i-e_j}]} \frac{(-1)^{d(e_i,k)+d(e_j,k)}}{4}\\
&=\sum_{k \in G\setminus{\{e_i,e_j\}}[P_{e_i-e_j}]} \frac{(-1)^{d(e_i,e_j)}}{4}
+\sum_{k \notin G\setminus{\{e_i,e_j\}}[P_{e_i-e_j}]} \frac{-(-1)^{d(e_i,e_j)}}{4} &(\text{by Lemma }\ref{parity of e_i,e_j}(a)(b))\\
&=\frac{(-1)^{d(e_i,e_j)}}{4}\left(|G\setminus{\{e_i,e_j\}}[P_{e_i-e_j}]|-(n-|G\setminus{\{e_i,e_j\}}[P_{e_i-e_j}|])\right)\\
&=\frac{(-1)^{d(e_i,e_j)}}{4}\left(2|G\setminus{\{e_i,e_j\}}[P_{e_i-e_j}]|-n\right).\\
\end{align*}

Case 2. $e_i,e_j \notin C$\\
In this case, $a_{i,k}a_{j,k}\neq 0$ only for all $k\in \big(G\setminus e_i(C))\cap (G\setminus e_j(C)\big)$ (where $P_{e_i-i^*} \subseteq P_{e_j-j^*}$ or $P_{e_j-j^*} \subseteq P_{e_i-i^*}$).
\begin{align*}
s^+_{i,j}&=\sum_{k\in V(G)} a_{i,k}a_{j,k} \\ 
&=\sum_{k\in \big(G\setminus e_i(C))\cap (G\setminus e_j(C)\big)} (-1)^{d(e_i,k)+d(e_j,k)} \\
&=\sum_{k\in \big(G\setminus e_i(C))\cap (G\setminus e_j(C)\big)} -(-1)^{d(e_i,e_j)} & \text{(by Lemma \ref{parity of e_i,e_j}(c))} \\
&=-(-1)^{d(e_i,e_j)}  |\big(G\setminus e_i(C))\cap (G\setminus e_j(C)\big)|.\\
\end{align*}

Case 3. $e_i \in C$ and $e_j \notin C$\\
In this case, $a_{i,k}a_{j,k}\neq 0$ only for all $k\notin G\setminus e_j[C]$.
\begin{align*}
s^+_{i,j}&=\sum_{k\in V(G)} a_{i,k}a_{j,k} \\ 
&=\sum_{k\notin G\setminus e_j[C]} \frac{(-1)^{d(e_i,k)+d(e_j,k)}}{2} \\
&=\sum_{k\notin G\setminus e_j[C]} \frac{-(-1)^{d(e_i,e_j)}}{2} & \text{(by Lemma \ref{parity of e_i,e_j}(d))} \\
&=\frac{-(-1)^{d(e_i,e_j)}}{2}  |G\setminus e_j(C)|\\
&=\frac{-(-1)^{d(e_i,e_j)}}{2}  \big(|G\setminus e_i(C)|+|G\setminus e_j(C)|\big) &\text{(since $G\setminus e_i(C)$ is empty)}.
\end{align*}

Case 4. $e_i \notin C$ and $e_j \in C$\\
This case has a proof similar to that of Case 3.
\end{proof}

\begin{corollary}
Let $G$ be an odd unicyclic graph on $n$ vertices and $n$ edges $e_1,e_2,\ldots,e_n$ with the cycle $C$ and the incidence matrix $M$. Suppose $S=M^TM$ is the signless edge-Laplacian of $G$ and its inverse is $S^{-1}=[s_{i,j}^+]$. Edge $e_i$ is in $C$ if and only if $s_{i,i}^+=\frac{n}{4}$.  If $e_i$ is a pendant edge, then $s_{i,i}^+=1$. The converse is also true when $n\neq 4$.
\end{corollary}

\begin{example}
For the odd unicyclic graph given in Figure \ref{fig:odd unicyclic}, 
$$S =\left[\begin{array}{rrrrrrr}
2 & 0 & 1 & 1 & 0 & 1 & 1 \\
0 & 2 & 0 & 0 & 1 & 0 & 1 \\
1 & 0 & 2 & 0 & 0 & 1 & 0 \\
1 & 0 & 0 & 2 & 1 & 0 & 1 \\
0 & 1 & 0 & 1 & 2 & 0 & 1 \\
1 & 0 & 1 & 0 & 0 & 2 & 0 \\
1 & 1 & 0 & 1 & 1 & 0 & 2
\end{array}\right]
\text{ and }
S^{-1}= \frac{1}{4}\left[\begin{array}{rrrrrrr}
12 & 0 & -4 & -6 & 6 & -4 & -6 \\
0 & 4 & 0 & 2 & -2 & 0 & -2 \\
-4 & 0 & 4 & 2 & -2 & 0 & 2 \\
-6 & 2 & 2 & 7 & -5 & 2 & 1 \\
6 & -2 & -2 & -5 & 7 & -2 & -3 \\
-4 & 0 & 0 & 2 & -2 & 4 & 2 \\
-6 & -2 & 2 & 1 & -3 & 2 & 7
\end{array}\right].$$
\end{example}

\section{Open Problems}
We found combinatorial formulas for the Moore-Penrose inverses of the incidence matrix $M$, signless Laplacian $Q$, and signless edge-Laplacian $S$ for a tree (a bipartite graph) and an odd unicyclic graph (a non-bipartite graph).  There are still a lot of problems open such as the following:

\begin{enumerate}
    \item Find the Moore-Penrose inverses $M^+$, $Q^+$, and $S^+$ for bipartite graphs, in particular, for (a) even unicyclic graphs, (b) even bicyclic graphs.
    
    \item Find the Moore-Penrose inverses $M^+$, $Q^+$, and $S^+$ for non-bipartite graphs, in particular, for (a) odd bicyclic graphs, (b) complete multipartite graphs such as $K_{n_1,n_2,n_3}$.
\end{enumerate}

\vspace{12pt}
\noindent {\bf Acknowledgments}\\
The authors would like to thank Dr. R.B. Bapat whose research inspired the authors to write this article.

\end{document}